\documentclass[12pt,a4paper]{article}
\usepackage{amsmath,amssymb,amsfonts}
\usepackage[latin1]{inputenc}

\def\Bbb{\mathbb}

\title{\bf Mean values and variances of the digits of $1/p$}

\author{Kurt Girstmair}

\date{}

\makeatletter
\let\@@maketitle=\maketitle
\def\maketitle{\def\thispagestyle##1{\relax}\@@maketitle}
\makeatother
\newtheorem{theorem}{Theorem}

\def\BE{\begin{equation}}
\def\EE{\end{equation}}
\def\BD{\begin{displaymath}}
\def\ED{\end{displaymath}}
\def\BA{\begin{array}}
\def\EA{\end{array}}
\def\BEA{\begin{eqnarray*}}
\def\EEA{\end{eqnarray*}}
\def\BI{\bibitem}

\def\Z{\Bbb Z}
\def\Q{\Bbb Q}

\def\phi{\varphi}

\def\MB{\mbox}
\def\LD{\ldots}
\def\OV{\overline}

\def\WH{\widehat}

\def\sminus{\smallsetminus}
\def\DIV{\,|\,}
\def\NDIV{\, \nmid \,}

\def\MN{\medskip\noindent}
\def\STOP{\hfill$\Box$}

\def\BPS{B_{\psi}}
\def\BPSB{B_{\OV{\psi}}}

\newcommand{\btop}[2]{\genfrac{}{}{0pt}{1}{#1}{#2}}

\def\LS#1#2{ \left( \frac{#1}{#2} \right) }

\begin{document}
\maketitle
\normalsize

\begin{abstract}

Let $p\ge 3$ be a prime and $b\ge 2$ an integer such that $p$ does not divide $b$. Then $1/p$ has a periodic digit expansion with respect to the basis $b$. The length $l$ of the period
is the (multiplicative) order of $b$ mod $p$. In the cases $l=p-1$ and $l=(p-1)/2$, formulas for the variance of the digits of a period were given previously.
These formulas involved Dedekind sums, class numbers of imaginary quadratic number fields, and generalized Bernoulli numbers.
In the present paper we develop a theory of this kind for $l=(p-1)/2^m$, $m\ge 1$, which
covers the special case $l=(p-1)/2$.

\end{abstract}

\section*{1. Introduction and results}

Let $p\ge 3$ be a prime, $b\ge 2$ an integer such that $p\NDIV b$.
Then
\BE
\label{1.1}
  \frac 1p=\sum_{j=1}^{\infty}c_jb^{-j},
\EE
where the numbers $c_j\in\{0,1,\LD,b-1\}$ are the {\em digits} of $1/p$ with respect to the basis $b$. It is well-known that the sequence of the digits is periodic and
that $(c_1,\LD,c_l)$ is a period, $l$ being the (multiplicative) order of $b$ mod $p$; see \cite{Gi1}.

Let $S$ and $T$ denote the sums
\BE
\label{1.2}
  S=\sum_{j=1}^l c_j,\: T=\sum_{j=1}^l c_j^2.
\EE
Hence the mean value of the digits $c_1,\LD,c_l$ is $S/l$.
The variance of this sequence is
\BD
 \sigma^2=\frac{T}{l}-\frac{S^2}{l^2}.
\ED
In the paper \cite{Gi2} we determined $\sigma^2$ in the case $l=p-1$
Indeed, in this case we obtained
\BD
\sigma^2=\frac{2bs(p,b)}{p-1}+\frac{(b-1)(bp-3b+p+3)}{12(p-1)},
\ED
where
\BE
\label{1.2.1}
  s(p,b)=\sum_{k=1}^{b-1}((k/b))((pk/b))
\EE
is the classical Dedekind sum; for its definition see \cite[formula (1)]{RaGr}.

In the paper \cite{Gi3} we determined the variance of the sequence $c_1,\LD,c_l$ in the case $l=(p-1)/2$.
We will see that this result is a special case of the results of this paper.

These results require some preparations, which we present now.
Let $\chi$ be a Dirichlet character mod $p$ of order
$2^r$, $r\ge 1$. Let $\zeta$ denote a primitive $2^r$th root of unity.
A linear combination
\BE
\label{1.3}
 \sum_{k=1}^{p-1}a_k\chi(k),\: a_k\in\Q,
\EE
can be expressed in terms of the standard basis $(1,\zeta,\LD,\zeta^{2^{r-1}-1})$ of the field $\Q(\zeta)$, namely,
\BD
\label{1.3.1}
 \sum_{k=1}^{p-1}a_k\chi(k)=\sum_{j=0}^{2^{r-1}-1}b_j\zeta^j,\: b_j\in\Q.
\ED
We put
\BE
\label{1.3.2}
  \left[\sum_{k=1}^{p-1}a_k\chi(k)\right]_0=b_0.
\EE
Note that this definition is independent of the choice of $\zeta$. Indeed, the character $\chi$ takes only the values $\pm\zeta^j$, $j=0,\LD, 2^{r-1}-1$, and $(1,\zeta,\LD,\zeta^{2^{r-1}-1})$ is $\Q$-linearly
independent. Accordingly,
\BD
   b_j=\sum_{\btop{k=1}{\chi(k)=\zeta^j}}^{p-1}a_k-\sum_{\btop{k=1}{\chi(k)=-\zeta^j}}^{p-1}a_k,\:j=0,\LD,2^{r-1}-1.
\ED
In particular,
\BD
  b_0=\sum_{\btop{k=1}{\chi(k)=1}}^{p-1}a_k-\sum_{\btop{k=1}{\chi(k)=-1}}^{p-1}a_k,
\ED
and so $b_0$ does not depend on $\zeta$.

For an odd Dirichlet character $\chi$ mod $p$ of the order $2^r$, $r\ge 1$, we need the generalized Bernoulli number
\BD
  B_{\chi}=\frac 1p\sum_{k=1}^{p-1}k\chi(k).
\ED
Here $\left[B_{\chi}\right]_0$ is well-defined.

Suppose that $l=(p-1)/2^m$, $m\ge 1$. In particular, $p\equiv 1$ mod $2^m$. Then we can select a Dirichlet character $\chi_r$ mod $p$ of order $2^r$, $r=1,\LD,m$.
The characters $\chi_1,\LD,\chi_{m-1}$ are even. The character $\chi_m$ is even, if $p\equiv 1$ mod $2^{m+1}$, and odd, if $p\equiv 2^m+1$ mod $2^{m+1}$.

\begin{theorem} 
\label{t1}

Let $l=(p-1)/2^m$ be the order of $b$ mod $p$. If $p\equiv 1$ mod $2^{m+1}$,
the sum $S$ of {\rm (\ref{1.2})} has the form
\BD
  S=(b-1)(p-1)/2^{m+1}.
\ED
 If $p\equiv 2^m+1$ mod $2^{m+1}$, then
\BD
 S=(b-1)(p-1)/2^{m+1}+(b-1)\left[B_{\chi_m}\right]_0/2,
\ED
where $\chi_m$ is a character mod $p$ of order $2^m$.

\end{theorem} 

Let $d$ be a positive integer. Then $X_d$ denotes the set Dirichlet characters mod $d$. The subsets of even and odd characters are denoted by $X_d^+$ and $X_d^-$, respectively.
For $\chi\in X_p$  and $\psi\in X_d$  the Dirichlet character $\chi\psi\in X_{pd}$ is defined by $\chi\psi(k)=\chi(k)\psi(k)$. The corresponding Bernoulli number
is
\BD
  B_{\chi\psi}=\frac 1{dp}\sum_{k=1}^{dp-1}k\chi\psi(k).
\ED
We use a bar to denote complex conjugation, for example $\OV{\psi}$ for the complex conjugate of the character $\psi$.
Recall that $s(p,b)$ is the classical Dedekind sum; see (\ref{1.2.1}). As usual, $\phi(\LD)$ denotes Euler's function.

\begin{theorem} 
\label{t2}

Let $l=(p-1)/2^m$ be the order of $b$ mod $p$, and, as above, $\chi_r$ a Dirichlet character mod $p$ of order $2^r$, $r=1,\LD,m$. The sum $T$ of {\rm (\ref{1.2})} takes the value
\BE
\label{1.4}
 T=(T_0+T_1)/2^m+[T_2]_0/2^{m-1}+\LD+[T_m]_0/2.
\EE
Here
\BE
\label{1.6}
 T_0=2bs(p,b)+\frac{b-1}6(2bp-p-3b+3)
\EE
and
\BE
\label{1.8}
   T_r=2b\sum_{d\DIV b}\frac{\OV{\chi_r}(d)}{\phi(d)}\sum_{\psi\in X_d^-}\OV{\psi}(p)B_{\chi_r\psi}B_{\OV{\psi}},
\EE
for $r=1,\LD,m-1$, and also for $r=m$ if $\chi_m$ is even.
If $\chi_m$ is odd, then
\BE
\label{1.10}
 T_m=(b-1)^2B_{\chi_m}.
\EE
\end{theorem} 

\MN
{\em Remarks.} 1. We will see that the number $T_r$ of (\ref{1.8}) is a rational linear combination of the character values $\chi_r(k)$, $k=1,\LD, p-1$.
Therefore, $[T_r]_0$ is well-defined; recall (\ref{1.3}), (\ref{1.3.2}).

2. The main results of the paper \cite{Gi3} immediately follow from the special case $m=1$ in Theorems \ref{t1}, \ref{t2}.

3. Let $\chi\in X_p^+$. If $\psi\in X_d^-$, $d\DIV b$, is a primitive character, then $B_{\chi\psi}$ and $\BPSB$ do not vanish. If $\psi$ is imprimitive, let $\psi'$ be the primitive character mod $f$, $f\DIV d$, that induces $\psi$.
Then
\BE
\label{1.12}
 \BPS=\prod_{q\DIV d}(1-\psi'(q))B_{\psi'} \MB{ and } B_{\chi\psi}=\prod_{q\DIV d}(1-\chi\psi'(q))B_{\chi\psi'},
\EE
where $q$ runs through the prime divisors of $d$; see \cite[p. 274]{Sz}. In this case the Bernoulli numbers $B_{\chi\psi}$ and $\BPSB$ may vanish.

4. If $b$ and $m$ are given, the set of primes $p$ such that $b$ has the order $(p-1)/2^m$ has a positive natural density in many cases, provided that the Generalized Riemann Hypothesis holds. Indeed, this is always true
if $b$ is square-free; see \cite{Mo}. For instance, if $b=10$ and $m=3$, this density is $27A/608$, where $A$ is Artin's constant
\BD
  A=\prod_q\left(1-\frac 1{q(q-1)}\right)=0.3739558\ldots,
\ED
$q$ running through all primes.

5. Connections between digits and class number factors (i.e., Bernoulli numbers or their products) have been investigated in several papers; see, for instance, \cite{Gi1, Hi, MuTh, ChKr, Mi, PuSa}.

\section*{2. Special cases}

In this section we consider the case $m=3$ for $b=10$ and $b=12$. In addition, we assume that $\chi_3$ is odd, i.e., $p\equiv 9$ mod $16$.
Whereas the relevance of $b=10$ is clear, the case $b=12$ is remarkable inasmuch as
the formula for $T_1$ involves class numbers of imaginary quadratic number fields.

First let $b=10$. The only divisors $d$ of $10$ such that $X_d^-$ is not empty are $d=5,10$.
Since $\chi_1$ is the Legendre symbol, namely, $\chi_1(k)=\LS kp$ for $k\in\Z$, we have $\OV{\chi_1}=\chi_1$.
Because $p\equiv 1$ mod $8$, $\chi_1(2)$ equals $1$, and since the order of $10$ mod $p$
is $(p-1)/8$, we see that $\chi_1(10)=\chi_2(10)=\chi_3(10)=1$. This implies $\chi_1(5)=1$.
Accordingly, $\chi_2(2),\chi_2(5)\in\{\pm 1\}$.
We have $X_5^-=\{\psi_5,\OV{\psi_5}\}$, where $\psi_5$ is defined by $\psi_5(2)=i$. Moreover, $X_{10}^-=\{\WH{\psi_5},\WH{\OV{\psi_5}}\}$, where $\,\WH{\: }\,$ denotes the character mod $10$ induced by the respective character mod $5$.

In the case $r=1$, formula (\ref{1.8}) reads
\BE
\label{2.2}
  T_1=10\left(\MB{Re}(\OV{\psi_5}(p)B_{\chi_1\psi_5}B_{\OV{\psi_5}})+\MB{Re}(\OV{\psi_5}(p)B_{\chi_1\WH{\psi_5}}B_{\WH{\OV{\psi_5}}})\right),
\EE
where Re$(\LD)$ denotes the real part of a complex number.
Here $B_{\psi_5}=(-3-i)/5$, and, by (\ref{1.12}), $B_{\chi_1\WH{\psi_5}}=(1-i)B_{\chi_1\psi_5}$, and $B_{\WH{\OV{\psi_5}}}=(1+i)B_{\OV{\psi_5}}$.
Therefore, we obtain
\BE
\label{2.4}
 T_1=6\,\MB{Re}(\OV{\psi_5}(p)B_{\chi_1\psi_5}(-3+i)).
\EE

In the case $r=2$ one has to observe that $\chi_2$ is no more a real character and that $\chi_2(5)$ may be $\pm 1$. Accordingly, the analogue of formula (\ref{2.2}) looks more complicated, namely,
\begin{eqnarray}
\label{2.6}
 T_2&=&5\chi_2(5)\left(\OV{\psi_5}(p)B_{\chi_2\psi_5}B_{\OV{\psi_5}}+\psi_5(p)B_{\chi_2\OV{\psi_5}}B_{\psi_5}\right)+ \nonumber\\
    && 5\left(\OV{\psi_5}(p)B_{\chi_2\WH{\psi_5}}B_{\WH{\OV{\psi_5}}}+\psi_5(p)B_{\chi_2\WH{\OV{\psi_5}}}B_{\WH{\psi_5}}\right).
\end{eqnarray}
Note that $\chi_2(2)$ may be $\pm 1$, so we have $B_{\chi_2\WH{\psi_5}}=(1-\chi_2(2)i)B_{\chi_2\psi_5}$ and the respective formula for $B_{\chi_2\WH{\OV{\psi_5}}}$\,.
By (\ref{1.10}),
\BE
\label{2.8}
T_3=81B_{\chi_3}.
\EE

\MN
{\em Example.} Let $p=1609$. Since $7$ is a primitive root mod $p$, we may define $\chi_3$  by $\chi_3(7)=\zeta$, $\zeta=e^{2\pi i/8}$.
Then $B_{\chi_3}=-23-3\zeta-25\zeta^2-3\zeta^3$ and $[B_{\chi_3}]_0=-23$. Theorem \ref{t1} gives $S=801$.

In order to compute $T$ by means of Theorem \ref{t2}, we observe
$\psi_5(p)=-1$ and $B_{\chi_1\psi_5}=22+18i$. So formula (\ref{2.4}) yields $T_1=504$. We define $\chi_2$ by $\chi_2(7)=i$. Then $\chi_2(5)=\chi_2(2)=1$.
We have $B_{\chi_2\psi_5}=14+14i$ and $B_{\chi_2\OV{\psi_5}}=18+30i$. Formula (\ref{2.6}) gives $T_2=240+408i$.
By (\ref{2.8}), $[T_3]_0=-81\cdot 23=-1863$.
Finally, we obtain $T_0=45804$ from formula (\ref{1.6}). Formula (\ref{1.4}) yields $T=4917$.

\MN
Now let $b=12$. The relevant divisors $d$ of $12$ are $d=3,4,6,12$. We have $X_3^-=\{\psi_3\}$ with $\psi_3(k)=\LS k3$, $X_4^-=\{\psi_4\}$ with $\psi_4(3)=-1$, $X_6^-=\{\WH{\psi_3}\}$,
and $X_{12}^-=\{\WH{\psi_3}, \WH{\psi_4}\}$. Here $\,\WH{\: }\,$ denotes the character induced by the respective character mod $3$ and mod $4$. This case is simpler than the case $b=10$,
inasmuch as all characters in $X_d^-$, $d=3,4,6,12$, are real. We have $B_{\psi_3}=-1/3$ and $B_{\psi_4}=-1/2$.

In the case $r=1$, we use $B_{\chi_1\psi_3}=-h(-3p)$, $B_{\chi_1\psi_4}=-h(-p)$, where $h(n)$ is the class number of $\Q(\sqrt n)$. Formula (\ref{1.8}) gives
\begin{eqnarray}
\label{2.10}
   T_1&=&4\chi_1(3)\psi_3(p)h(-3p)+6\psi_4(p)h(-p)+16\chi_1(6)\psi_3(p)h(-3p)+\nonumber\\
       &&8\psi_3(p)h(-3p)+6\psi_4(p)(1+\chi_1(3))h(-p).
\end{eqnarray}
In the same way,
\begin{eqnarray}
\label{2.12}
   T_2&=&-4\OV{\chi_2}(3)\psi_3(p)B_{\chi_2\psi_3}-6\OV{\chi_2}(4)\psi_4(p)B_{\chi_2\psi_4}-8\OV{\chi_2}(6)\psi_3(p)(1+\chi_2(2))B_{\chi_2\psi_3}-\nonumber\\
       &&4\psi_3(p)(1+\chi_2(2))B_{\chi_2\psi_3}-6\psi_4(p)(1+\chi_2(3))B_{\chi_2\psi_4}.
\end{eqnarray}
Finally,
\BE
\label{2.14}
  T_3=121B_{\chi_3}.
\EE
{\em Example.} Let $p=601$. Since $7$ is a primitive root mod $p$, we may define $\chi_3$  by $\chi_3(7)=\zeta$, $\zeta=e^{2\pi i/8}$.
Then $[B_{\chi_3}]_0=-15$. Theorem \ref{t1} gives $S=330$. We have $\chi_1(3)=\chi_1(6)=1$, further $\psi_3(p)=\psi_4(p)=1$. The class numbers of formula (\ref{2.10})  are $h(-3p)=8$ and $h(-p)=20$.
We obtain $T_1=584$ from (\ref{2.10}). As concerns $T_2$, we define $\chi_2$ by $\chi_2(7)=i$. We have $\chi_2(2)=\chi_2(3)=\chi_2(6)=1$, and $B_{\chi_2\psi_3}=-8-8i$, $B_{\chi_2\psi_4}=-4+8i$.
Thus, $T_2=296+80i$, by (\ref{2.12}). Moreover, $[T_3]_0=-1815$, by (\ref{2.14}). Finally, $T_0=25300$, by (\ref{1.6}). So formula (\ref{1.4}) gives $T=2402$.

\section*{3. Proofs}

Let $b\ge 2$ and $p$ be as above, in particular $p\NDIV b$ and the order of $b$ mod $p$ is $l=(p-1)/2^m$, $m\ge 1$.
For an integer $j$ let $(j)_p$ be the representative of $j$ in $\{0,\LD,p-1\}$, i.e., $(j)_p$ is the integer $k$, $0\le k\le p-1$, that satisfies $k\equiv j$ mod $p$.
Then the digit $c_j\in\{0,\LD b-1\}$ of $1/p$ is given by
\BD
   c_j=\frac{b(b^{j-1})_p-(b^j)_p}p,
\ED
$j=1,\LD,l$; see \cite{Gi1}.
For $k=0,\LD,b$, let
\BD
   I_{b,k}=\Z\cap(0,kp/b).
\ED
In particular, $I_{b,0}=\emptyset$ and $I_{b,b}=\{1,\LD,p-1\}$. It is easy to see that, for $j=1,\LD,l$, and $k=0,\LD,b-1$,
we have $c_j\le k$ if, and only if $(b^{j-1})_p\in I_{b,k+1}$.
Again, we choose a Dirichlet character $\chi_r$ mod $p$ of the order $2^r$, $r=1,\LD,m$; see Theorem \ref{t2}.
Since $(b^{j-1})_p$ runs through the $2^m$th powers mod $p$,
\BE
\label{3.2}
  |\{j; 1\le j\le l, c_j\le k\}|=|\{n\in I_{b,k+1}; \chi_m(n)=1\}|.
\EE
For $k=0,\LD,b$, and $r=1,\LD,m$, put
\BD
  t_{r,b,k}=\sum_{n\in I_{b,k}}\chi_r(n).
\ED
and
\BE
\label{3.6}
  S_r=\sum_{k=0}^{b-1}k\cdot(t_{r,b,k+1}-t_{r,b,k}).
\EE
Note that $S_r$ has the form (\ref{1.3}).
In addition, put
\BE
\label{3.8}
 S_0=\sum_{k=0}^{b-1}k\cdot |I_{b,k+1}\sminus I_{b,k}|.
\EE

\MN
{\em Proof of Theorem \ref{t1}.}
For $r\ge 1$,
\BE
\label{3.9.1}
  [S_r]_0=\sum_{k=0}^{b-1}k\cdot(|\{n\in I_{b,k+1}\sminus I_{b,k};\chi_r(n)=1\}|-|\{n\in I_{b,k+1}\sminus I_{b,k};\chi_r(n)=-1\}|)
\EE
and $[S_1]_0=S_1$.
We have
\BE
\label{3.10}
   S=\sum_{k=0}^{b-1}k\cdot|\{j; 1\le j\le l, c_j=k\}|=\sum_{k=0}^{b-1}k\cdot|\{n\in I_{b,k+1}\sminus I_{b,k}; \chi_m(n)=1\}|,
\EE
by (\ref{3.2}).
Now
\BD
 (S_0+S_1)/2=\sum_{k=0}^{b-1}k\cdot|\{n\in I_{b,k+1}\sminus I_{b,k};\chi_1(n)=1\}|.
\ED
We use the fact that $\chi_r(n)=1$ if, and only if, $\chi_{r+1}(n)=\pm 1$, $r\ge 1$. Indeed, $\chi_r(n)=1$ is the same as saying $\chi_{r+1}(n^2)=1$, which means $\chi_{r+1}(n)^2=1$.
Thereby, and by (\ref{3.9.1}),
\BD
 ((S_0+S_1)/2+[S_2]_0)/2=\sum_{k=0}^{b-1}k\cdot|\{n\in I_{b,k+1}\sminus I_{b,k};\chi_2(n)=1\}|.
\ED
We use this argument repeatedly and obtain, in view of (\ref{3.10}),
\BE
\label{3.12}
 (S_0+S_1)/2^m+[S_2]_0/2^{m-1}+\LD+[S_m]_0/2=\sum_{k=0}^{b-1}k\cdot|\{n\in I_{b,k+1}\sminus I_{b,k}; \chi_m(n)=1\}|=S.
\EE
As concerns $S_0$, we observe
\BD
  |I_{b,k}|=\begin{cases} \lfloor kp/b\rfloor & \MB{ if } k\le b-1;\\
                      p-1                 & \MB{ if } k=b.
         \end{cases}
\ED
Hence the telescoping nature of formula (\ref{3.8}) implies
\BE
\label{3.12.1}
 S_0=(b-1)(p-1)-\sum_{k=1}^{b-1}\lfloor kp/b\rfloor =(b-1)(p-1)-(b-1)(p-1)/2=(b-1)(p-1)/2;
\EE
for the sum on the left-hand side see \cite[formula (41)]{RaGr}.

Because of its telescoping nature, formula (\ref{3.6}) yields, for $r\ge 1$,
\BE
\label{3.13}
  S_r=(b-1)t_{r,b,b}-\sum_{k=1}^{b-1}t_{r,b,k}=-\sum_{k=1}^{b-1}t_{r,b,k}
\EE
since $t_{r,b,b}=\sum_{n=1}^{p-1}\chi_r(n)=0$.
Now suppose $(k,b)=d$ for a number $k\in\{1,\LD,b-1\}$. Then
\BD
  t_{r,b,k}=t_{r,b/d,k/d}\: \MB{ with }\: b/d>1, (b/d,k/d)=1.
\ED
Therefore,
\BE
\label{3.14}
 \sum_{k=1}^{b-1}t_{r,b,k}=\sum_{\btop{d\DIV b,}{d>1}}\sum_{\btop{l=1}{(l,d)=1}}^{d-1}t_{r,d,l}.
\EE

Suppose that $\chi_r$ is even. Then formula (6) of \cite{Sz} says, since $d>1$ and $(l,d)=1$,
\BE
\label{3.18}
  t_{r,d,l}=\sum_{1\le n< lp/d}\chi_r(n)=-\frac{\OV{\chi_r}(d)}{\phi(d)}\sum_{\psi\in X_d^-}\OV{\psi}(lp)B_{\chi_r\psi}.
\EE
From (\ref{3.14}) and (\ref{3.18}) we obtain
\BE
\label{3.19}
  \sum_{k=1}^{b-1}t_{r,b,k}=-\sum_{\btop{d\DIV b,}{d>1}}\frac{\OV{\chi_r}(d)}{\phi(d)}\sum_{\psi\in X_d^-}\OV{\psi}(p)B_{\chi_r\psi}\sum_{\btop{l=1}{(l,d)=1}}^{d-1}\OV{\psi}(l).
\EE
Here the innermost sum on the right-hand side is $0$, since $\OV{\psi}$ is not the principal character. Because of (\ref{3.13}), $S_r=0$.

Suppose that $\chi_m$ is odd. Then the said formula (6) of \cite{Sz} yields
\BD
  t_{m,d,l}=-B_{\chi_m}+\frac{\OV{\chi_m}(d)}{\phi(d)}\sum_{\psi\in X_d^+}\OV{\psi}(lp)B_{\chi_r\psi}.
\ED
This results in
\BE
\label{3.19.1}
 \sum_{k=1}^{b-1}t_{m,b,k}=-\sum_{\btop{d\DIV b,}{d>1}}\sum_{\btop{l=1}{(l,d)=1}}^{d-1}B_{\chi_m}+
 \sum_{\btop{d\DIV b,}{d>1}}\frac{\OV{\chi_m}(d)}{\phi(d)}\sum_{\psi\in X_d^+}\OV{\psi}(p)B_{\chi_m\psi}\sum_{\btop{l=1}{(l,d)=1}}^{d-1}\OV{\psi}(l).
\EE
The initial double sum on the right-hand side equals
\BD
 -\sum_{\btop{d\DIV b,}{d>1}}\phi(d)B_{\chi_m}=-(b-1)B_{\chi_m}.
\ED
The following triple sum boils down to
\BE
\label{3.20}
  \sum_{\btop{d\DIV b,}{d>1}}\OV{\chi_m}(d)B_{\chi_m\psi_{0, d}},
\EE
where $\psi_{0,d}$ denotes the principal character mod $d$.
By (\ref{1.12}), we have $B_{\chi_m\psi_{0, d}}=\prod_{q\DIV d}(1-\chi_m(q))B_{\chi_m}$,
$q$ running through the prime divisors of $d$. Now the sum of formula (\ref{3.20}) vanishes, since
\BD
  \sum_{d\DIV b}\OV{\chi_m}(d)\prod_{q\DIV d}(1-\chi_m(q))=1.
\ED
Indeed, this expression equals
\BD
   \sum_{d\DIV b}\OV{\chi}_m(d)\sum_{t\DIV d}\chi_m(t)\mu(t)=\sum_{t\DIV b}\mu(t)\sum_{u\DIV \frac bt}\OV{\chi_m}(u)=\sum_{u\DIV b}\OV{\chi_m}(u)\sum_{t\DIV \frac bu}\mu(t)=\OV{\chi_m}(b)=1,
\ED
because $b$ is a $2^m$th power mod $p$ (as usual, $\mu(\LD)$ denotes M\"obius' function). In view of (\ref{3.13}), we obtain $S_m=(b-1)B_{\chi_m}$.
Theorem \ref{t1} follows from formula (\ref{3.12}).
\STOP

\MN
{\em Proof of Theorem 2.} The proof is quite analogous to the above proof. Indeed, we have
\BD
   T=\sum_{k=0}^{b-1}k^2\cdot|\{j; 1\le j\le l, c_j=k\}|=\sum_{k=0}^{b-1}k^2\cdot|\{n\in I_{b,k+1}\sminus I_{b,k}; \chi_m(n)=1\}|,
\ED
We define
\BD
  T_r=\sum_{k=0}^{b-1}k^2\cdot(t_{r,b,k+1}-t_{r,b,k})
\ED
and
\BD
 T_0=\sum_{k=0}^{b-1}k^2\cdot |I_{b,k+1}\sminus I_{b,k}|;
\ED
see (\ref{3.6}), (\ref{3.8}). This gives the analogue of (\ref{3.12}), namely,
\BD
 (T_0+T_1)/2^m+[T_2]_0/2^{m-1}+\LD+[T_m]_0/2=\sum_{k=0}^{b-1}k^2\cdot|\{n\in I_{b,k+1}\sminus I_{b,k}; \chi_m(n)=1\}|=T.
\ED
Now the telescoping sum $T_0$ can be written
\BD
   T_0=(b-1)^2(p-1)-2\sum_{k=1}^{b-1}k\lfloor kp/b\rfloor+\sum_{k=1}^{b-1}\lfloor kp/b\rfloor,
\ED
where the second sum on the right-hand side equals $(b-1)(p-1)/2$; see (\ref{3.12.1}). The first sum takes the value
\BD
  -bs(p,b)+(b-1)(4bp-2p-3b)/12;
\ED
see \cite[formula (39)]{RaGr}. Altogether, we have the identity (\ref{1.6}).

For $r\ge 1$ the telescoping sum $T_r$ takes the form
\BE
\label{3.30}
 T_r=-2\sum_{k=1}^{b-1}kt_{r,b,k}+\sum_{k=1}^{b-1}t_{r,b,k}.
\EE
The second of these sums has been computed in the proof of Theorem \ref{t1}. Its value is $0$ if $\chi_r$ is even and $-(b-1)B_{\chi_m}$ if $\chi_m$ is odd; see (\ref{3.13}) and the end of the said proof.

Let $\chi_r$ be even. Then the evaluation of the first sum leads to an analogue of formula (\ref{3.19}), i.e.,
\BD
  \sum_{k=1}^{b-1}kt_{r,b,k}=-\sum_{\btop{d\DIV b,}{d>1}}\frac bd\,\frac{\OV{\chi_r}(d)}{\phi(d)}\sum_{\psi\in X_d^-}\OV{\psi}(p)B_{\chi_r\psi}\sum_{\btop{l=1}{(l,d)=1}}^{d-1}l\OV{\psi}(l).
\ED
We obtain (\ref{1.8}) since the second sum of (\ref{3.30}) vanishes (note that the condition $d>1$ can be omitted, because $X_1^-=\emptyset$).

Let $\chi_m$ be odd. Then the evaluation of the first sum of (\ref{3.30}) leads to an analogue of formula (\ref{3.19.1}), namely,
\BD
  \sum_{k=1}^{b-1}kt_{m,b,k}=-\sum_{\btop{d\DIV b}{d>1}}\frac bd\sum_{\btop{l=1}{(l,d)=1}}^{d-1}lB_{\chi_m}+
  \sum_{\btop{d\DIV b}{d>1}}\frac bd\,\frac{\OV{\chi}_m(d)}{\phi(d)}\sum_{\psi\in X_d^+}\OV{\psi}(p)B_{\chi_m\psi}\sum_{\btop{l=1}{(l,d)=1}}^{d-1}l\OV{\psi}(l).
\ED
In order to evaluate the initial double sum on the right-hand side, we use the formula
\BD
  \sum_{\btop{l=1}{(l,d)=1}}^{d-1}l=d\phi(d)/2;
\ED
see \cite[p. 48]{Ap}.
Therefore, the value of this double sum is $-b(b-1)B_{\chi_m}/2$. The following triple sum vanishes. Indeed, its innermost sum is $0$ except if $\psi=\psi_{0,d}$, the principal character mod $d$.
In this case the innermost sum is $d\phi(d)/2$ and the value of the triple sum is
\BD
  \frac b2\sum_{\btop{d\DIV b,}{d>1}}\OV{\chi_m}(d)B_{\chi_m\psi_{0,d}}.
\ED
But this expression vanishes, as we have seen in the case of the sum of formula (\ref{3.20}).
Altogether, we obtain the identity (\ref{1.10}).
\STOP

\bigskip
\centerline{\bf Competing interests and data availability}

\MN
The author declares that there are no competing interests. The paper has no associated data.


\MN
Institut f\"ur Mathematik \\
Universit\"at Innsbruck   \\
Technikerstr. 13/7        \\
A-6020 Innsbruck, Austria \\
Kurt.Girstmair@uibk.ac.at

\end{document}